\documentclass[final]{siamltex}
\usepackage{epsfig}
\usepackage{amsmath}
\usepackage{epic}
\usepackage{comment}
\usepackage{amssymb}
\usepackage{graphicx}
\usepackage{epsfig}
\usepackage{amssymb}
\usepackage{amsmath}
\usepackage{amsfonts}
\usepackage{amsbsy}
\usepackage{mathrsfs}
\usepackage{setspace}

\newcommand{\innprd}[2]{\left< #1 , #2 \right>}
\newcommand{\abs}[1]{\lvert#1\rvert}
\def\op#1{{\mathcal #1}}
\newcommand{\RR}{\mathbb{R}}
\newcommand{\A}{\mathcal{A}}
\newcommand{\Lo}{\mathcal{L}}
\newcommand{\T}{\mathscr{T}}
\newcommand{\U}{\mathcal{U}}
\newcommand{\p}{\mathcal{P}}
\newcommand{\uu}{\vec{u}}
\newcommand{\xb}{\mathbf{x}}
\newcommand{\ub}{\mathbf{u}}
\newcommand{\fb}{\mathbf{f}}
\newcommand{\pb}{\mathbf{p}}
\newcommand{\myphi}{\vec{\varphi}}
\newcommand{\di}{\mathrm{div}\,}
\newcommand{\e}{\vec{e}}
\newcommand{\ff}{\vec{f}}
\newcommand{\gv}{\vec{g}}
\newcommand{\bnorm}[1]{|\kern-0.12em|\kern-0.12em|#1|\kern-0.12em|\kern-0.12em|}
\newcommand{\R}{\mathbb{R}}

\newcommand{\junk}[1]{{}}

\newtheorem{remark}[theorem]{Remark}

\title{Multigrid solution of a distributed optimal control problem constrained 
       by the {S}tokes equations}

\author{Andrei Dr{\u{a}}g{\u{a}}nescu\thanks{Department 
    of Mathematics and Statistics, University of Maryland, Baltimore
    County, 1000~Hilltop Circle, Baltimore, Maryland 21250 ({\tt draga@umbc.edu}). The work of this author was supported 
    in part by the Department of Energy under contract no. DE-SC0005455, and by the National Science Foundation under award
    DMS-1016177.}
  \and{Ana Maria Soane\thanks{Department 
    of Mathematics and Statistics, University of Maryland, Baltimore
    County, 1000~Hilltop Circle, Baltimore, Maryland 21250 ({\tt asoane@umbc.edu}). The work of this author was supported 
    in part by the Department of Energy under contract no. DE-SC0005455.}}}

\begin{document}

\maketitle

\begin{abstract} In this work we construct multigrid preconditioners to accelerate the solution process of a 
linear-quadratic optimal control problem constrained by the Stokes system. 
The first order optimality conditions of the control problem form  a linear system (the KKT system)
connecting the state, adjoint, and control variables. Our approach is to 
eliminate the state and adjoint variables by essentially solving two Stokes systems, 
and to construct efficient multigrid preconditioners for the 
Schur-complement of the block associated with the state and adjoint variables.
These multigrid preconditioners are shown to be of optimal order
with respect to the convergence properties of the discrete methods used to
solve the Stokes system. In particular, the number of conjugate gradient iterations
is shown to decrease as the resolution increases, a feature shared by similar
multigrid preconditioners for elliptic constrained optimal control problems.
\end{abstract}



\pagestyle{myheadings}
\thispagestyle{plain}
\markboth{A.~DR{\u A}G{\u A}NESCU AND A.~SOANE}{MULTIGRID FOR STOKES-CONSTRAINED OPTIMIZATION}

\section{Introduction}
\label{sec:intro}

Over the last decade, the computational  community has shown  a growing 
interest in devising fast solution methods for large-scale distributed optimal 
control problems constrained by  partial differential equations (PDEs). 
Optimal control problems constrained by the Stokes system form a  stepping 
stone in the  natural progression from the -- now classical -- 
Poisson-constrained test problem to  problems constrained by more specialized 
and complex  PDE systems  modeling fluid flow such as Navier-Stokes,
non-Newtonian Stokes, or the shallow water equations. Optimal control problems 
constrained by such  PDE models play important roles in real-world 
applications, such as modeling of ice sheets or data assimilation for ocean 
flows and weather models. 

In this article we consider an optimal control problem with a cost functional
of tracking-type:
\begin{align}\label{equ:cost}
 \min_{\uu, p, \ff}\  J(\uu,p,\ff) = 
       \frac{\gamma_u}{2}\|\uu - \uu_d\|^2_{L_2(\Omega)^2} +
       \frac{\gamma_p}{2}\|p -p_d\|^2_{L_2(\Omega)} + 
       \frac{\beta}{2}\|\ff\|^2_{L_2(\Omega)^2}\ ,
\end{align}
subject to the constraints
\begin{equation}\label{equ:stokes}
\left\{\begin{aligned}
 -\Delta \uu + \nabla p &= \ff \quad \text{in } \Omega, \\
                \di \uu &= 0   \quad \text{in } \Omega, \\
                    \uu &= 0   \quad \text{on } \partial \Omega,
\end{aligned}\right .
\end{equation}
where $\Omega$ is a  bounded polygonal domain in $\R^2$. The purpose of the
control problem is to identify a force $\ff$ that gives rise to a velocity 
$\uu$ and/or pressure~$p$ to match a known target velocity $\uu_d$, 
respectively pressure $p_d$.
Since this problem is ill-posed, we consider a standard Tikhonov 
regularization for the force with the regularization parameter $\beta$ being a fixed positive number. The constants $\gamma_u$, $\gamma_p$ are nonnegative, 
not both zero.
   
The main goal of this article is to construct and analyze optimal order
multigrid preconditioners to be used in the solution process 
of~\eqref{equ:cost}-\eqref{equ:stokes}. Over the last few  years a 
significant amount of work has been devoted to developing multigrid methods 
for optimal control problems. An overview of this research and further 
references can be found in a review article by Borzi and Schulz~\cite{BS}. 
However, fewer works are dedicated specifically to  optimal control problems 
constrained by the Stokes system.
For example, recently,  Rees and Wathen~\cite{RW} have proposed two 
preconditioners for the optimality system in the distributed control of the 
Stokes system, a block-diagonal preconditioner for MINRES and a 
block-lower triangular preconditioner for a nonstandard conjugate 
gradient method. 
We note  that there are several papers in the literature on finite element 
error analysis for the optimal control of the Stokes equations 
(see, e.g., \cite{NS,RV,RMV} and the references therein).  
Our paper focuses on the solution of the linear system that arises in the 
discretization process, which is not addressed in these papers.
However, for completeness, we also prove an a priori error estimate for the 
optimal control since the cost functional in \eqref{equ:cost} 
includes a pressure term, 
using standard techniques similar to those used in \cite{NS,RV,DD}.

Since the cost functional in~\eqref{equ:cost} is quadratic, the KKT system is 
a linear saddle-point problem connecting the state, adjoint, and control 
variables. Solution methods for these problems typically fall into two 
categories: the all-at-once approach takes advantage of the sparsity of the
system, but has the disadvantage that the matrix  is indefinite. 
On the other hand, Schur-complement strategies may lead to smaller systems that may also be positive definite, but the sparsity is lost. Our approach is to 
eliminate the state and adjoint variables by essentially solving two Stokes 
systems using specific methods (see~\cite{MR2155549}), and to construct 
efficient multigrid preconditioners for the Schur-complement of the block 
associated to the state and adjoint variables.
The constructed multigrid preconditioners are related to the ones developed by
Dr{\u a}g{\u a}nescu and Dupont~\cite{DD}, and  are shown to be of optimal 
order with respect to the convergence properties of the discrete finite 
element methods used to solve the Stokes system. In particular, the number of 
conjugate gradient iterations is shown to decrease as the resolution increases, a feature shared by similar multigrid preconditioners for elliptic-constrained 
optimal control problems. One word on the optimality of the preconditioner: 
the usual notion of optimality, especially in the context of multigrid, refers 
to the cost of the solution process being proportional to the number of 
variables. We argue that for this problem, an unpreconditioned application of 
conjugate gradient (CG) in conjunction with an optimal multigrid solve for the 
Stokes system already satisfies  this notion of optimality. In the current 
context, multigrid preconditioners actually can perform better than that, and 
optimality refers to the order of approximation of the operator under 
scrutiny -- in this case the reduced Hessian -- by the multigrid 
preconditioner, as shown in Theorem~\ref{thm:prec}.

The paper is organized as follows: 
In Section \ref{sec:conv},  we introduce the discrete optimal control problem  
and prove finite element estimates that will be needed for the multigrid 
analysis. Section \ref{sec:prec} contains the main result of the paper, 
Theorem~\ref{thm:prec}, which refers to the analysis of the two-grid 
preconditioner; furthermore, the extension to multigrid preconditioners is 
briefly discussed. In Section~\ref{sec:num}  we present numerical experiments 
that illustrate our theoretical results, and we formulate some conclusions in
Section~\ref{sec:concl}.

\section{Discretization and convergence results} \label{sec:conv}
The strategy we adopt is to first discretize the optimal control problem then
optimize the discrete problem. To define a discrete problem based on 
a finite element approximation of the Stokes system we briefly recall the usual weak formulation of the  Stokes equations. Define the spaces
\begin{align*}
 V &= H_0^1(\Omega)^2, \\
 M &= L_{2,0}(\Omega) = \bigg\{ p \in L_2(\Omega): \int_{\Omega} p = 0\bigg\}, 
\end{align*}
and the bilinear forms
$a: V \times V \rightarrow \R$ and $b:V \times M \rightarrow \R$ as
\begin{align*}
  a(\uu, \myphi) &= \sum_{i=1}^2 \int_{\Omega}
  \nabla \uu_i \cdot \nabla \myphi_i, \\
  b(\myphi, p)   &= -\int_{\Omega} p \, \di \myphi.
\end{align*}
Throughout this paper, we write $(\cdot,\cdot)$ for the inner product in 
$L_2(\Omega)$ or~$L_2(\Omega)^2$ according to context, and similarly   
for the norms, if there is no risk of misunderstanding. 

The weak solution $(\uu,p) \in V \times M$ of \eqref{equ:stokes} is the 
solution of
\begin{align*}
 a(\uu, \myphi) + b(\myphi,p) &= (\ff, \myphi)  &\forall \myphi \in V, \\ 
 b(\uu, \psi)                 &= 0   &\forall \psi \in M.
\end{align*} 
For $\ff \in H^{-1}(\Omega)^2$ the problem has a unique solution \cite{GR}. 
Moreover, if $\Omega$ is a convex polygon and $\ff \in L_2(\Omega)^2$, then 
$\uu \in H^2(\Omega)^2$, $p \in H^1(\Omega)$ \cite{KeoS},  
and there exists $C=C(\Omega)>0$ such that 
\begin{align}\label{equ:regc}
 \|\uu\|_{H^2(\Omega)^2} + \|\nabla p\|_{L_2(\Omega)} 
 \leq C\|\ff\|_{L_2(\Omega)^2}.
\end{align}
Throughout this paper we will assume $\Omega$ to be convex, so that the 
$H^2$- regularity of the Stokes problem is ensured.
Furthermore, the target velocity field $\uu_d$ is assumed to be 
from $L_2(\Omega)^2$ and the target pressure $p_d$ from $M$.

We introduce the solution mappings $\,\U$ and $\p$ of the continuous state 
equation defined  such that for any $\ff \in L_2(\Omega)^2$ the following holds:
\begin{align*}
 a(\U\ff, \myphi) + b(\myphi, \p \ff)  = (\ff, \myphi) 
 \quad \text{and} \quad b(\U \ff,\psi) = 0 \quad 
 \forall (\myphi, \psi) \in V \times M. 
\end{align*}
The mapping $\U$, considered as a linear operator in $L_2(\Omega)^2$, is 
compact and self-adjoint, as 
\begin{align*}
 (\U \ff_1, \ff_2) = a(\U \ff_1, \U \ff_2) = (\ff_1, \U \ff_2) \quad 
 \forall \ff_1, \ff_2 \in L_2(\Omega)^2.
\end{align*}
We denote by $\p^*:L_{2,0}(\Omega) \rightarrow L_2(\Omega)^2$ the adjoint 
operator of $\p$, defined by  
\begin{align*}
 (\p^*q, \ff)_{L_2(\Omega)^2} = (q, \p \ff)_{L_2(\Omega)} \quad
 \forall q \in L_{2,0}(\Omega),  \ff \in L_2(\Omega)^2. 
\end{align*}
With this notation, the  problem~\eqref{equ:cost}-\eqref{equ:stokes} is written in reduced, unconstrained form as
\begin{align}\label{equ:costred}
 \min_{\ff\in L_2(\Omega)^2}\  \hat{J}(\ff) = 
 \frac{\gamma_u}{2}\|\U\ff - \uu_d\|^2_{L_2(\Omega)^2} +
 \frac{\gamma_p}{2}\|\p \ff -p_d\|^2_{L_2(\Omega)} + 
 \frac{\beta}{2}\|\ff\|^2_{L_2(\Omega)^2}\ .
\end{align}
The Hessian operator associated to the reduced cost functional $\hat{J}$ is 
given by
$$H_{\beta} = \gamma_u \, \U^*\U  + \gamma_p \p^* \p + \beta I\ .$$
Note that the solution of the minimization problem \eqref{equ:cost} is 
obtained as the solution of the normal equation  
\begin{align}\label{equ:sol_cont}
 H_{\beta} \ff = \gamma_u \,\U^* \uu_d + \gamma_p \p^* p_d.
\end{align}
The goal of this paper is to design an efficient multigrid algorithm for 
solving the discrete version of \eqref{equ:sol_cont}.

\subsection{Finite element approximation}
\label{ssec:feapprox}

We consider a shape regular  quasi-uniform quadrilateral mesh 
$\T_h$ of $\bar{\Omega}$, and we assume that the mesh $\T_h$  results from 
a coarser regular mesh $\T_{2h}$ from one uniform refinement. We use the 
Taylor-Hood $\mathbf{Q}_2-\mathbf{Q}_1$ finite elements 
to discretize the state equation. 
The velocity field $\uu$ is approximated in the space 
$V_h^0 = V_h \cap H_0^1(\Omega)^2$, where   
\begin{align*}
 V_h   &= \{v_h \in C(\bar{\Omega})^2: v_h|_T \in Q_2(T)^2 
          \text{ for } T \in \T_h\},
\end{align*}
and the pressure $p$ is approximated in the space
\begin{align*}
 M_h   &= \{q_h \in C(\Omega) \cap L_{2,0}(\Omega): q_h|_T \in Q_1(T) 
          \text{ for } T \in \T_h\},
\end{align*}
where $Q_k(T)$ is the space of polynomials of degree less than or equal 
to~$k$ in each variable~\cite{Ci}.  
The control variable $\ff$ is approximated by continuous piecewise 
biquadratic polynomial vectors in the space $V_h$.  
\begin{remark}
For convenience, we choose here the quadrilateral $\mathbf{Q}_2-\mathbf{Q}_1$ 
Taylor-Hood elements, however, our analysis can be extended to triangular  
$\mathbf{P}_2-\mathbf{P}_1$ elements as well as other stable mixed 
finite elements.
\end{remark}

For a given control $\ff \in L_2(\Omega)^2$, the solution 
$(\uu_h,p_h) \in V_h^0 \times M_h$ of the discrete state equation 
is given by 
\begin{align*}
 a(\uu_h, \myphi_h) + b(\myphi_h,p_h) &= (f, \myphi_h)  
                                      &\forall \myphi_h \in V_h^0, \\
 b(\uu_h, \psi_h) &= 0                &\forall \psi_h \in M_h.
\end{align*}
Let $\U_h$ and $\p_h$ be the solution mappings of the discretized state 
equation and $\U_h^*$, $\p_h^*$ their adjoints, defined analogously to  
the continuous counterparts. 
Furthermore, denote by $\pi_h: L_2(\Omega)^2 \rightarrow V_h$ the 
$L_2$-orthogonal projection onto~$V_h$.
The discretized, reduced optimal control problem reads
\begin{equation}\label{equ:discrete}
\begin{aligned}
 \min_{\ff_h \in V_h} \hat{J}_h(\ff_h) = 
 \frac{\gamma_u}{2}\|\U_h \ff_h - \uu_d^h \|^2 +
 \frac{\gamma_p}{2}\|\p_h \ff_h -p_d^h\|^2 +  
 \frac{\beta}{2} \|\ff_h\|^2,
\end{aligned}
\end{equation}
where $\uu_d^h$, $p_d^h$ are the $L_2$-projections of 
the data onto $V_h$, respectively $M_h$. 

Let us investigate the structure of the algebraic system associated to the 
discretized optimal control problem. 
Let $\{\myphi_j\}_{j=1}^p$  and $\{\psi_k\}_{k=1}^m$  
be the basis functions of the spaces $V_h$ and $M_h$, respectively. Furthermore,
assume that  $\{\myphi_j\}_{j=1}^n$ are the basis functions of $V^0_h$ for some $n<p$.
We expand the discrete solutions $\uu_h$ and $p_h$ as
\begin{align*}
 \uu_h(x) = \sum_{j=1}^n \ub_j \myphi_j(x), \quad
 p_h(x) = \sum_{k=1}^m \pb_k \psi_k(x), 
\end{align*}
and we approximate the control by 
\(
 \ff_h(x)= \sum_{j=1}^p \fb_j \myphi_j(x).
\)
Let  $\mathbf{A} \in \RR^{n \times n}$ and $\mathbf{B} \in \RR^{m \times n}$ 
be the matrices related to the bilinear forms,   
$\mathbf{A} = [a_{ij}] = a(\myphi_i, \myphi_j)$,
$\mathbf{B} = [b_{ij}] = b(\myphi_i, \psi_j)$. Furthermore
consider the mass-matrix \mbox{$\mathbf{M}_{\ff} = [\int_{\Omega} \myphi_i\cdot \myphi_j]\in \RR^{p \times p}$} and its
submatrices $\mathbf{M}_{\uu\ff}=\mathbf{M}_{\ff}(1:n,1:p)\in \RR^{n \times p}$ and  \mbox{$\mathbf{M}_{\uu}=\mathbf{M}_{\ff}(1:n,1:n)\in \RR^{n \times n}$}.
The  
discrete state equation takes the form
\begin{align*}
 \begin{bmatrix}
  \mathbf{A} & \mathbf{B}^T \\
  \mathbf{B} & 0
 \end{bmatrix}
 \begin{bmatrix} \ub \\ \pb \end{bmatrix} =
 \begin{bmatrix} \mathbf{M}_{\uu\ff} \\ 0 \end{bmatrix} \fb.
\end{align*}
After discretizing, the reduced cost functional becomes
\begin{align*}
 \hat{J}_h = \frac{\gamma_u}{2} (\ub-\ub_d)^T\mathbf{M}_{\uu} (\ub - \ub_d)+
       \frac{\gamma_p}{2} (\pb-\pb_d)^T\mathbf{M}_p (\pb-\pb_d) + 
       \frac{\beta}{2} \fb^T \mathbf{M}_{\ff} \fb,
\end{align*}
where  $\mathbf{M}_p = [\int_{\Omega}\psi_i\psi_j]$ and $\ub_d$, $\pb_d$
are the coefficient vectors in the expansions of $\uu_h^d$ and $p_h^d$  
in $V^0_h$, respectively $M_h$.

Let us introduce the matrices   
\begin{align*}
 \mathbf{S} =
 \begin{bmatrix}
  \mathbf{A} & \mathbf{B}^T \\
  \mathbf{B} & 0
 \end{bmatrix},\ 
 \mathbf{M} =
 \begin{bmatrix}
  \gamma_u \mathbf{M}_{\uu} & 0 \\
   0 & \gamma_p \mathbf{M}_p
 \end{bmatrix}, \ 
 \mathbf{L} =
 \begin{bmatrix}
  \mathbf{M}_{\uu\ff} \\ 0
 \end{bmatrix}, \  
 \text{and} \quad
 \xb_d = 
 \begin{bmatrix}
  \ub_d \\ \pb_d 
 \end{bmatrix}.
\end{align*}
After dropping the constant terms  the problem becomes
\begin{align*}
 \underset{\fb}{\min}\  \frac{1}{2} \fb^T \big( 
 \mathbf{L}^T \mathbf{S}^{-1}\mathbf{M}\mathbf{S}^{-1}\mathbf{L} + 
 \beta \mathbf{M}_{\ff}\big) \fb - 
 \fb^T \mathbf{L}^T \mathbf{S}^{-1} \mathbf{M}\, \xb_d ,
\end{align*}
which reduces to solving the linear system
\begin{align}\label{equ:sys}
 \big(\mathbf{L}^T \mathbf{S}^{-1}\mathbf{M}\mathbf{S}^{-1}\mathbf{L} + 
 \beta \mathbf{M}_{\ff}\big) \fb = 
 \mathbf{L}^T \mathbf{S}^{-1} \mathbf{M}\, 
 \xb_d  .
\end{align}
Note that the system matrix is dense, thus \eqref{equ:sys} has to be 
solved using iterative methods, and for increased efficiency we need 
high-quality preconditioners. In Section~\ref{sec:prec}, we construct and 
analyze a multigrid preconditioner for the system  
\begin{align}\label{equ:sysr}
\big(\beta \mathbf{I} + \mathbf{M}_{\ff}^{-1}\mathbf{L}^T \mathbf{S}^{-1}
                        \mathbf{M}\mathbf{S}^{-1}\mathbf{L}\big) \fb = 
 \mathbf{M}_{\ff}^{-1}\mathbf{L}^T \mathbf{S}^{-1} \mathbf{M}\, 
 \xb_d\ ,
\end{align}
which is obtained from~\eqref{equ:sys} by left-multiplying 
with $\mathbf{M}_{\ff}^{-1}$.

We should remark that the system~\eqref{equ:sysr} can be obtained from the 
KKT system associated with the discrete constrained optimal control problem 
associated to~\eqref{equ:cost}-\eqref{equ:stokes} by block-eliminating 
the velocity, pressure, and the  Lagrange multipliers. 
So~\eqref{equ:sysr} is in fact a reduced KKT system. 

\subsection{Estimates and convergence results}
We first  list some results on the Stokes equations and their numerical 
approximation  which are needed for the multigrid analysis. 
\begin{theorem}\label{cond:assume}
There exist constants $C_1=C_1(\U,\Omega)$ and 
$C_2=C_2(\p,\Omega)$ such that the following hold:
\begin{enumerate}
 \item[(a)] smoothing:
  \begin{align}\label{cond:smoothu}
   \|\U\ff\| \leq C_1 \|\ff\|_{H^{-2}(\Omega)^2} \quad 
   \forall \ff \in L_2(\Omega)^2, 
  \end{align}
  and
  \begin{align}\label{cond:smoothp}
   \|\p \ff \| \leq C_2 \|\ff \|_{H^{-1}(\Omega)^2} \quad 
   \forall \ff \in L_2(\Omega)^2; 
  \end{align}
 \item[(b)] approximation:
  \begin{align}\label{cond:approxu}
   \| \U\ff - \U_h\ff\| \leq C_1 h^2 \|\ff\| \quad \forall \ff \in 
   L_2(\Omega)^2, 
  \end{align}
  and
  \begin{align}\label{cond:approxp}
   \| \p\ff - \p_h\ff\| \leq C_2 h \|\ff\| \quad \forall \ff \in 
   L_2(\Omega)^2; 
  \end{align}
 \item[(c)] stability:
  \begin{align}\label{cond:stab}
   \|\U_h \ff\| \leq C_1 \|\ff\| \ \text{ and } \  
   \|\p_h \ff\| \leq C_2 \| \ff\| \quad \forall \ff \in L_2(\Omega)^2.
  \end{align}
\end{enumerate}
\end{theorem} 
\begin{proof}

\noindent
\begin{enumerate}
\item[(a)]
  The inequality~\eqref{cond:smoothu} is a straightforward consequence 
  of \eqref{equ:regc} 
 (see \cite[Corollary 6.2]{DD}), while~\eqref{cond:smoothp} follows immediately from Brezzi's splitting theorem, 
 see, e.g.,
 \cite[Theorem 4.3]{Br}.
\item[(b)]
 This is a standard approximation result, see, e.g.,  \cite{Ste}.
\item[(c)]
Follows immediately from  the estimates in (a) and (b). 
\end{enumerate}    
\end{proof}
We state without proof the following well-known result.
\begin{lemma} \label{lemma:id-proj}
The following 
approximation of the 
 identity by the projection holds:
  \begin{align}\label{cond:proj}
   \|(I-\pi_h)\ff\|_{H^{-k}(\Omega)^2} \leq C h^k \|\ff\| \quad 
   \forall \ff\in L_2(\Omega)^2, 
  \end{align}
  for $k=0,1,2,3$, with $C$ constant independent of $h$. 
\end{lemma}
\begin{remark}\label{rem:H}
Theorem \ref{cond:assume} and Lemma \ref{lemma:id-proj} imply that 
there  are constants $C_1 = C_1(\Omega, \U)$ and 
$C_2 = C_2(\Omega, \p)$ such that
 \begin{align}\label{equ:u} 
  \|\U(I- \pi_h)\ff\| \leq C_1 h^2 \|\ff\| \quad 
  \forall \ff \in  L_2(\Omega)^2  
 \end{align}
 and
 \begin{align}\label{equ:p} 
  \|\p(I- \pi_h)\ff\| \leq C_2 h \|\ff\| \quad 
  \forall \ff \in  L_2(\Omega)^2.  
 \end{align}
\end{remark}
\begin{lemma}
Let $\tilde{\pi}_h:L_{2,0}(\Omega) \rightarrow M_h$ be the $L_2$-orthogonal 
projection onto~$M_h$. Then 
\begin{align}\label{cond:projp}
 |((I -\tilde{\pi}_h) q, \p \ff)| \leq C h \|q\| \, \|\ff\| 
 \quad \forall q \in L_{2,0}(\Omega), \ff \in L_2(\Omega)^2, 
\end{align}
with $C$ constant independent of $h$. 
\end{lemma}
\indent{\em Proof.}
We have 
\begin{align*}
 |((I-\tilde{\pi}_h)q,\p\ff)| &= |((I-\tilde{\pi}_h)q,\p\ff-\p_h\ff)| 
 \leq \|(I-\tilde{\pi}_h)q\| \, \|\p\ff -\p_h\ff\| \\
 &\overset{\eqref{cond:approxp}}{\leq} C h \|q\| \, \|\ff\| \quad 
  \forall q \in L_{2,0}(\Omega), \ff \in L_2(\Omega)^2.\qquad \square
\end{align*}
Denote
\[
G_u = \U^* \U\ ,\ \ G_p = \p^* \p\ ,\ \ G_u^h = \U^*_h \U_h\ ,\ \  
G_p^h = \p^*_h \p_h \ .
\]
\begin{lemma}\label{lemma:piH}
 The following approximation properties hold:
 \begin{align}\label{equ:piHu}
  \|\pi_h(G_u^h - G_u)\ff \| \leq C h^2 \|\ff\| \quad \forall \ff \in V_h, 
 \end{align}
 and
 \begin{align}\label{equ:piHp}
  \|\pi_h(G_p^h - G_p)\ff \| \leq C h \|\ff\| \quad \forall \ff \in V_h, 
 \end{align}
 for some constant 
$C$ independent of $h$.
\end{lemma}
\begin{proof}
For $\ff \in V_h$ we have
\begin{align*}
        |(\pi_h(G_u^h - G_u)\ff,\ff)| 
 &=     |\|\U_h\ff\|^2 - \|\U\ff\|^2| \\
 &\leq \|\U_h\ff-\U\ff\|(\|\U_h\ff\| + \|\U\ff\|) \\
 &\leq  C h^2\|\ff\|^2, 
\end{align*}
and \eqref{equ:piHu} follows from the symmetry of $\pi_h(G_u^h - G_u)$.

Similarly, we have
\begin{align*}
    |(\pi_h(G_p^h - G_p)\ff,\ff)| 
 &= |(G_p^h \ff,\ff)_V-(G_p\ff,\ff)_V| \\
 &= |(\p_h \ff,\p_h\ff)_M - (\p\ff,\p\ff)_M| \\
 &= |\|\p_h\ff\|^2 - \|\p\ff\|^2| \\
 &\leq \|\p_h\ff-\p\ff\|(\|\p_h\ff\| + \|\p\ff\|) \\
 &\leq C h\|\ff\|^2 \quad  \forall \ff \in V_h ,
\end{align*}
and \eqref{equ:piHp} follows from the symmetry of $\pi_h(G_p^h - G_p)$.
\end{proof}

From the definition of $H_{\beta}$ follows
\begin{align}\label{equ:condno}
 \beta \|\ff\|^2 \leq ( H_{\beta}\ff,\ff) \leq (\beta + C) \|\ff\|^2  \quad 
 \forall \ff \in L_2(\Omega)^2,
\end{align}
with $C=C(\U, \p, \Omega)$,  
and a similar estimate holds for the discrete Hessian
\begin{align}\label{equ:dhess}
 H_{\beta}^h = \beta I + \gamma_u G_u^h + \gamma_p G_p^h,
\end{align}
which shows that the condition number of $H_{\beta}^h$ is bounded uniformly with respect to $h$, but
potentially increasing with $\beta\downarrow 0$.
The inequality~\eqref{equ:condno} also implies that the cost functional $\hat{J}$ is strictly convex  and has a 
unique minimizer given by 
\begin{align*}
 \ff^{min} = H_{\beta}^{-1}(\gamma_u \U^*\uu_d + \gamma_p\p^*p_d).
\end{align*}
Similarly, the minimizer of the discrete quadratic is
\begin{align*}
 \ff_h^{min} = (H^h_{\beta})^{-1}(\gamma_u \U_h^*\uu_d^h + 
                                  \gamma_p\p_h^*p_d^h).
\end{align*}
In the following theorem, we show that $\ff_h^{min}$ approximates 
$\ff^{min}$ to optimal order in the $L_2$-norm. 
\begin{theorem}
There exists a constant $C = C(\Omega,\U,\p)$ independent of $h$ such that for
$h \leq h_0(\beta,\Omega, \U,\p)$ we have the following 
stability and error estimates:
\begin{eqnarray}
  \label{equ:opts}
  \|\ff^{\min}_h\|& \leq & \|\ff^{\min}\| + 
  \frac{C}{\beta}\big(\gamma_u h^2\|\uu_d\| +\gamma_p h \|p_d\|\big),\\ 
  \label{equ:opterr}
  \|\ff_h^{\min} - \ff^{\min}\|& \leq &\frac{C}{\beta}
 \big(\gamma_u h^2( \|\uu_d\| + \|\ff^{min}\|) + 
      \gamma_p h (\|p_d\| + \|\ff^{min}\| )\big) .
\end{eqnarray}
\end{theorem}
\begin{proof}
Let $\e_h = \ff_h^{min} - \ff^{min}$. We have
\begin{align*}
    H_{\beta}\e_h
 &= (\beta I + \gamma_u G_u + \gamma_p G_p) \ff_h^{min} - H_{\beta}\ff^{min} \\
 &= (\beta I + \gamma_u G_u + \gamma_p G_p) \ff_h^{min} 
    -\gamma_u \U^*\uu_d -\gamma_p \p^* p_d \\
 &= (\beta I + \gamma_u G_u^h + \gamma_p G_p^h)\ff_h^{min} +
     \gamma_u(G_u - G_u^h)\ff_h^{min} +
     \gamma_p(G_p - G_p^h)\ff_h^{min} \\ 
 &   \quad -\gamma_u \U^*\uu_d -\gamma_p \p^* p_d \\
 &=  \gamma_u \U_h^* \uu_d^h + \gamma_p \p_h^* p_d^h +
     \gamma_u(G_u - G_u^h)\ff_h^{min} +
     \gamma_p(G_p - G_p^h)\ff_h^{min} \\
 &   \quad -\gamma_u \U^*\uu_d -\gamma_p \p^* p_d \\
 &=  \gamma_u \underset{\A_u}{\underbrace{\big[ (G_u - G_u^h)\ff_h^{min} +
     (\U_h^* - \U^*)\uu_d^h - \U^*(I-\pi_h)\uu_d \big]}} \\
 &   \quad + \gamma_p \underset{\A_p}{\underbrace{
     \big[(G_p - G_p^h)\ff_h^{min}+(\p_h^* - \p^*)p_d^h - 
          \p^*(I-\tilde{\pi}_h)p_d\big]}},
\end{align*}
where we used that $\uu^h_d = \pi_h \uu_d$ and $p^h_d=\tilde{\pi}_h p_d$.
Furthermore, 
\begin{align*}
 \beta\|\e_h\|^2 \overset{\eqref{equ:condno}}{\leq}(H_\beta \e_h, \e_h)
 \leq \gamma_u (\A_u,\e_h) + \gamma_p (\A_p,\e_h),
\end{align*}
and
\begin{align*}
       (\A_u,\e_h)  
 &\overset{\eqref{cond:approxu}, \eqref{equ:u}}{\leq}
   C h^2 \|\uu_d\| \|\e_h\| + ((G_u - G_u^h)\ff_h^{min},\e_h) \\
 &\mspace{-15.0mu} \overset{(I-\pi_h)e_h \perp V_h}{=} 
   C h^2 \|\uu_d\| \|\e_h\| + (\pi_h(G_u -G_u^h)\ff_h^{min}, \pi_h \e_h) \\
 & \qquad + (G_u\ff_h^{min}, (I-\pi_h)\e_h)\\
 & \quad \overset{\eqref{equ:piHu}}{\leq} 
   C h^2\|\e_h\|(\|\uu_d\| + \|\ff_h^{min}\|)+(\U\ff_h^{min}, \U(I-\pi_h)\e_h)\\
 & \quad \overset{\eqref{equ:u}}{\leq} 
   C h^2\|\e_h\|(\|\uu_d\| + \|\ff_h^{min}\|),  
\end{align*}
with $C=C(\U,\Omega)$, where we  have also used the fact that $\U$ is 
self-adjoint.
Similarly, we get
\begin{align*}
   (\A_p,\e_h) 
 &= (p_d^h, (\p_h -\p)\e_h) - ((I-\tilde{\pi}_h)p_d, \p \e_h) +
    ((G_p - G_p^h)\ff_h^{min}, \e_h)\\
 &\overset{\eqref{cond:approxp}, \eqref{cond:projp}}{\leq}
    C h \|p_d\|\|\e_h\| + ((G_p - G_p^h)\ff_h^{min}, \e_h)\\   
 &\mspace{-15.0mu} \overset{(I-\pi_h)\e_h \perp V_h}{=} 
    C h \|p_d\| \|\e_h\| + (\pi_h(G_p -G_p^h)\ff_h^{min}, \pi_h \e_h) \\
 & \qquad + (G_p\ff_h^{min}, (I-\pi_h)\e_h)\\
 & \quad \overset{\eqref{equ:piHp}}{\leq} 
    C h\|\e_h\|(\|p_d\| + \|\ff_h^{min}\|) + (\p\ff_h^{min}, \p(I-\pi_h)\e_h)\\
 & \quad \overset{\eqref{equ:p}}{\leq} 
    C h\|\e_h\|(\|p_d\| + \|\ff_h^{min}\|),   
\end{align*}
with $C=C(\p,\Omega)$.
These estimates yield 
\begin{align*}
 \|\ff_h^{min}-\ff^{min}\| \leq \frac{C}{\beta}
 \Big(\gamma_u h^2 (\|u_d\| + \|\ff_h^{min}\|) +
      \gamma_p h   (\|p_d\| + \|\ff_h^{min}\| \Big).
\end{align*}
Since 
\begin{align*}
       \|\ff_h^{min}\| 
 &\leq \|\ff^{min}\| + \|\ff_h^{min} -\ff^{min}\|\\ 
 &\leq \|\ff^{min}\| + \frac{C}{\beta}
       \Big(\gamma_u h^2 (\|u_d\| + \|\ff_h^{min}\|) +
            \gamma_p h   (\|p_d\| + \|\ff_h^{min}\| \Big),
\end{align*}
we obtain \eqref{equ:opts} and \eqref{equ:opterr} for 
$h$ sufficiently small. 
\end{proof}

\section{Two-grid and multigrid preconditioner for the discrete Hessian}
\label{sec:prec}
In this section, we use the multigrid techniques developed in \cite{DD}  
to construct and analyze a two-level symmetric preconditioner for the 
discrete Hessian $H_{\beta}^h$ defined in \eqref{equ:dhess}. The extension 
to a multigrid preconditioner follows the same strategy as in \cite{DD} and
is further explained in great detail in~\cite{DP}.

For the remainder of this paper we consider on $V_h$ the Hilbert-space 
structure inherited from $L_2(\Omega)^2$. Furthermore, we consider the 
$L_2$-orthogonal decomposition $V_h = V_{2h}\oplus W_{2h}$ and let 
$\pi_{2h}$ be the $L_2$-projector onto $V_{2h}$. 
The analysis in \cite{DD} suggests that $H_{\beta}^h$ is well 
approximated by
\begin{align*}
  T_{\beta}^h \ \overset{\text{def}}{=} \ 
 H_{\beta}^{2h} \pi_{2h} + \beta(I-\pi_{2h}).
\end{align*}
For completeness we briefly recall the heuristics leading to the definition 
of~$T_{\beta}^h$. As usual in the multigrid literature, for 
$\ff_h\in V_h$ we regard $\pi_{2h}\ff_h$ as the ``smooth''  component of $\ff_h$, and 
$(I-\pi_{2h})\ff_h$ as the ``rough'' or ``oscillatory'' component; so the projector
$(I-\pi_{2h})$ extracts the ``oscillatory''  part of a function in  $V_h$.
If we write $H_{\beta}^{h} = H_{0}^{h} + \beta I$ and take into account the
``smoothing'' properties of $H_{0}^{h}$ (these are due to the compactness of the operator
$H_{0}$ which $H_{0}^{h}$ approximates), it follows that  the products $(I-\pi_{2h})H_{0}^{h}$ and 
$H_{0}^{h}(I-\pi_{2h})$ are almost negligible. So
\begin{eqnarray}
\nonumber
H_{\beta}^{h} &=& (\pi_{2h}+(I-\pi_{2h}))(H_{0}^{h} + \beta I)(\pi_{2h}+(I-\pi_{2h}))\\
\label{eq:twogridconstr}
&\approx &
\pi_{2h} (H_{0}^{h}+ \beta I)\pi_{2h}+\beta (I-\pi_{2h})\ .
\end{eqnarray}
Furthermore, when applied to the  ``smooth'' component $\pi_{2h} \ff_h$ of a function $\ff_h$,
it is expected that $H_{0}^h\pi_{2h} \ff_h \approx H_{0}\pi_{2h} \ff_h \approx H_{0}^{2h}\pi_{2h} \ff_h$,
hence the idea to  replace in~\eqref{eq:twogridconstr} $H_{0}^{h}$ by $H_{0}^{2h}$, which gives 
rise to  $T_{\beta}^h$.

Since $\pi_{2h}$ is a projection, 
$(T_{\beta}^h)^{-1}$ is computed explicitly as
\begin{align*}
  L_{\beta}^h\  \overset{\text{def}}{=}\  (T_{\beta}^h)^{-1} = 
 (H_{\beta}^{2h})^{-1}\pi_{2h} + \beta^{-1}(I-\pi_{2h}). 
\end{align*}
We propose $L_{\beta}^h \in \mathcal{L}(V_h)$ as a two-level preconditioner 
for $H_{\beta}^h$. 
To assess the quality of the preconditioner we use the spectral distance 
introduced in \cite{DD}, defined for two symmetric positive definite operators 
$T_1, T_2 \in \Lo(V_h)$ as 
\begin{align}
 d_h(T_1,T_2) = \underset{w \in V_h \setminus \{0\}}{\sup}
                \Bigg|\ln \frac{(T_1 w,w)}{(T_2 w,w)}\Bigg|.
\end{align}
If $L_{\beta}^h$ is a preconditioner for $H_{\beta}^h$ then 
the spectral radius $\rho(I-L_{\beta}^h H_{\beta}^h)$, 
which is an accepted quality-measure 
for a preconditioner, is controlled by the spectral distance between 
$L_{\beta}^h$ and $(H_{\beta}^h)^{-1}$ (see Lemma \ref{lemma:app} 
in  Appendix \ref{sec:app} for a 
precise formulation). 
The advantage of using the spectral distance over 
$\rho(I-L_{\beta}^h H_{\beta}^h)$ is that the former is a true distance 
function.
\begin{theorem}\label{thm:prec}
For $h < h_0(\beta,\Omega,\U,\p)$ there exists a constant
$C=C(\Omega,\U,\p)$
such that
\begin{align}\label{equ:dist_prec}
  d_h(H_{\beta}^h,T_{\beta}^h) \leq \frac{C}{\beta}(\gamma_u h^2 + \gamma_p h).
\end{align}
\end{theorem}
\indent{\em Proof.}
 We have
 \begin{equation}\label{equ:diff}
 \begin{aligned}
     T_{\beta}^h - H_{\beta}^h 
  &= H_{\beta}^{2h} \pi_{2h} + 
     \beta(I-\pi_{2h}) - H_{\beta}^h \\
  &= (\beta I + \gamma_u G_u^{2h} + \gamma_p G_p^{2h}) \pi_{2h} +
     \beta(I-\pi_{2h}) - (\beta I + \gamma_u G_u^h + \gamma_p G_p^h)\\
  &= \gamma_u(G_u^{2h}\pi_{2h} - G_u^h) + 
     \gamma_p(G_p^{2h}\pi_{2h} - G_p^h). 
\end{aligned}
\end{equation}
We write
\begin{align*}
 G_u^{2h}\pi_{2h} - G_u^h = (G_u^{2h} - G_u) \pi_{2h} +
                            (G_u - G_u^h) \pi_{2h} +
                            G_u^{h}(\pi_{2h} - I).
\end{align*}
For any $\gv \in V_h$ we have 
\begin{align*}
    |((G_u^{2h}-G_u)\gv,\gv)| 
 &= |(\U_{2h}^*\U_{2h}-\U^*\U)\gv,\gv)| = |\|U_{2h}\gv\|^2-\|\U\gv\|^2| \\
 &\leq \|(\U_{2h}-\U)\gv \| \ (\|\U_{2h}\gv\| + \|\U\gv\|) 
 \overset{\eqref{cond:approxu}}{\leq} C h^2\|\gv\|^2,
\end{align*}
which implies 
\begin{align*}
 \|(G_u^{2h}-G_u)\gv \|\leq C h^2 \|\gv\| 
\end{align*}
since $G_u^{2h}-G_u$ is symmetric on $V_h$. 
In particular, 
\begin{align*}
 \|(G_u^{2h}-G_u)\pi_{2h}\ff \|\leq C h^2 \|\ff\| \quad \forall \ff \in V_h.
\end{align*}
Similarly, it can be shown that
\begin{align*}
 \|(G_u-G_u^h)\pi_{2h}\ff\| \leq C h^2\|\ff\|. 
\end{align*}
Finally, we estimate
\begin{align*}
 \|G_u^h(I-\pi_{2h})\ff\| &= \|\U_h^* \U_h(I-\pi_{2h})\ff\| \leq 
 C \|\U_h(I-\pi_{2h})\ff\| \\ &\leq
 C \|\U(I-\pi_{2h})\ff\| + C \|(\U-\U_h)(I-\pi_{2h})\ff\| \\
 &\overset{\eqref{equ:u}, \eqref{cond:approxu}}{\leq}
 C_1 h^2 \|\ff\| + C_2h^2\|(I-\pi_{2h})\ff\| \leq C h^2\|\ff\|, 
 \quad \forall \ff \in V_h.
\end{align*}
Combining the last three estimates, we obtain 
\begin{align*}
 \|(G_u^{2h}\pi_{2h} -G_u) \ff \| \leq C h^2 \|\ff\| \quad \forall \ff \in V_h.
\end{align*}
The second term in \eqref{equ:diff} is estimated in a similar way, to obtain
\begin{align*}
 \|(G_p^{2h}\pi_{2h} -G_p) \ff \| \leq C h \|\ff\| \quad \forall \ff \in V_h, 
\end{align*}
which together with the previous estimate yields 
\begin{align*}
 \|(T_{\beta}^h - H_{\beta}^h)\ff\| \leq C(\gamma_u h^2 + \gamma_p h) \|\ff\|
 \quad \forall \ff \in V_h. 
\end{align*}
It follows that 
\begin{align*}
         \Bigg|\frac{(T_{\beta}^h\ff,\ff)}{(H_{\beta}^h\ff,\ff)} -1 \Bigg|
 &\leq  C\frac{(\gamma_u h^2 + \gamma_p h) \|\ff\|^2}
              {\gamma_u(H_u^h \ff,\ff) + \gamma_p (H_p^h \ff,\ff) + 
               \beta \|\ff\|^2} 
  \leq   \frac{C}{\beta}(\gamma_u h^2 + \gamma_p h).
\end{align*}
Assuming $C \beta^{-1}(\gamma_u h_0^2 + \gamma_p h_0) = \alpha <1$, 
and $0<h\leq h_0$ we conclude that $T_{\beta}^h$ is symmetric positive definite, and
\begin{align*}
  \underset{\ff \in V_h\setminus\{0\}}{\sup} 
  \Bigg|\ln \frac{(T_{\beta}^h\ff,\ff)}{(H_{\beta}^h\ff,\ff)} \Bigg|
 &
  \leq \frac{|\ln(1-\alpha)|}{\alpha}
  \underset{\ff \in V_h \setminus \{0\}}
  \sup \Bigg|\frac{(T_{\beta}^h\ff,\ff)}{(H_{\beta}^h\ff,\ff)} -1 \Bigg|\\
 &\leq \frac{|\ln(1-\alpha)|}{\alpha}
       \frac{C}{\beta}(\gamma_u h^2 + \gamma_p h)\\
 &\leq \frac{C}{\beta}(\gamma_u h^2 + \gamma_p h), 
\end{align*}
where we also used that for $\alpha \in (0,1)$, 
$x \in [1-\alpha,1+\alpha]$ we have
\begin{align*}
 \frac{\ln(1+\alpha)}{\alpha} |1-x| \leq |\ln x| \leq
 \frac{|\ln(1-\alpha)|}{\alpha}|1-x|.\qquad \square
\end{align*}

A consequence of Theorem \ref{thm:prec}, that legitimizes the use of 
$L_{\beta}^h$ as a preconditioner for the Hessian, is stated in the following 
corollary. 
\begin{corollary}
There exists a constant $C=C(\Omega,\U,\p)$ such that 
\begin{align}\label{equ:dist_prec2}
 d_h(L_{\beta}^h, (H_{\beta}^h)^{-1}) \leq 
 \frac{C}{\beta} (\gamma_u h^2 + \gamma_p h) .
\end{align}
for $h \leq h_0(\beta,\Omega,\U,\p)$. 
\end{corollary}

Note that, for fixed $\beta$, since the spectral distance between the operators decreases with $h\downarrow 0$, the quality of the preconditioner 
increases with $h\downarrow 0$; this  is different from classical multigrid preconditioners 
for elliptic problems, where the spectral distance is bounded above by a constant that is independent of $h$.

\junk{
The details of how to extend the two-grid preconditioner to a multigrid preconditioner
that exhibits the same optimal-order quality~\eqref{equ:dist_prec2} can be found in~\cite{DP}. We should just remark that
the classical V-cycle idea does not yield satisfactory results, and that the optimal order multigrid preconditioner has
structure similar to a classical W-cycle. Also, another important difference between this and the classical multigrid is that here
the base case needs to be chosen sufficiently fine, whereas in classical multigrid it can be as coarse as possible, in general.
} 

Next, we briefly describe how to extend the two-grid preconditioner to a  
multigrid preconditioner that exhibits the same optimal-order 
quality~\eqref{equ:dist_prec2} and is less costly to apply.
If we use the classical V-cycle idea to define recursively a multigrid 
preconditioner, the resulting preconditioner is suboptimal, that is, the quality 
of the preconditioner does not improve as $h\downarrow 0$, 
it is simply mesh-independent. 
To construct an improved preconditioner we introduce the operators
\begin{align*}
  \mathcal{G}_h:\mathcal{L}(V_{2h}) \rightarrow \mathcal{L}(V_h), 
  \quad \mathcal{G}_h(T) = T \pi_{2h} + \beta^{-1}(I - \pi_{2h})
\end{align*}
and
\begin{align*}
 \mathcal{N}_{h} : \mathcal{L}(V_h) \rightarrow \mathcal{L}(V_h), \quad
 \mathcal{N}_h(X) = 2X - X H_{\beta}^h X.  
\end{align*}
Note that 
$\mathcal{N}_h$ is related to the Newton iterator 
for the equation $X^{-1} - H^h_{\beta}=0$, i.e., $\mathcal{N}_h(X_0)$ is
the first Newton iterate starting at $X_0$. 
Thus, if $X_0$ is a good approximation for $(H^h_{\beta})^{-1}$
then $\mathcal{N}_h(X_0)$ is significantly closer to $(H_{\beta}^h)^{-1}$ than 
$X_0$. 
We follow \cite{DP} and define the multigrid preconditioner $K_{\beta}^h$ 
using the following algorithm:
\newpage
\begin{enumerate}
 \item 
   \verb+if+ \emph{coarsest level}
    \[ K_{\beta}^h = (H_{\beta}^h)^{-1} \]
 \item \verb+else+
    \begin{description}
    \item 
     \verb+if+ \emph{intermediate level}
     \[ K_{\beta}^h = \mathcal{N}_h(\mathcal{G}_h(K_{\beta}^{2h}))\]
    \item 
     \verb+else+ \% \emph{finest level}
     \[ K_{\beta}^h = \mathcal{G}_h(K_{\beta}^{2h})\]
    \item 
     \verb+end if+
     \end{description}
 \item \verb+end if+
\end{enumerate}
Since the application of $\mathcal{N}_h$ requires a matrix-vector
multiplication  by $H_{\beta}^h$, which for large scale problems is expected
to be very costly at the finest level, we prefer that no such matrix-vector 
multiplication  
is computed inside the preconditioner. 
This is the reason why 
we treat the cases of intermediate and 
finest resolutions differently.
\junk{
Note that the application of $\mathcal{N}_h$ requires a matrix-vector 
multiplication  by $H_{\beta}^h$, which for large scale problems is expected 
to be very costly at the finest level. We prefer that no such matrix-vector 
multiplication is computed inside the preconditioner, therefore (that is 
why) we treat the cases of intermediate and finest resolutions differently.  
}
In practice, neither $H_{\beta}^h$ nor $K_{\beta}^h$ is ever formed, so
both are applied matrix-free (see \cite{DP} for details).

The analysis 
of the multigrid preconditioner relies on the 
estimate for the two-grid preconditioner and properties of the spectral 
distance. 
We recall here Lemma $5.3$ from \cite{DD} that is needed for 
the analysis. 
\begin{lemma}\label{lemma:err}
Let $(e_i)_{i\geq 0}$ and $(a_i)_{i\geq 0}$ be positive numbers satisfying the 
recursive inequality 
\begin{align*}
 e_{i+1} \leq C(e_i + a_{i+1})^2
\end{align*} 
and
\begin{align*}
 a_{i+1} \leq a_i \leq f^{-1} a_{i+1}
\end{align*}
for some $0<f<1$. If $a_0\leq \frac{f}{4C}$ and if $e_0\leq 4Ca_0^2$, then 
\begin{align*}
 e_i \leq 4Ca_i^2, \quad \forall i>0.
\end{align*}
\end{lemma}

\begin{theorem}\label{thm:multi}
Assume that $C\beta^{-1}(\gamma_uh_0^2+\gamma_p h_0) < 2^{-5}$, where 
$C$ is the constant  from 
Theorem~\ref{thm:prec} and $K_{\beta}^{h_0} = (H_{\beta}^{h_0})^{-1}$. 
Then 
\begin{align}\label{equ:dist_prec_multi}
  d_h\big(K_{\beta}^h, (H_{\beta}^h)^{-1}\big) \leq 
  2\frac{C}{\beta}(\gamma_u h^2 + \gamma_p h), \quad 
 \text{for } h = 2^{-l}h_0, \, l \geq 2.
\end{align}
\end{theorem}
\begin{proof}
Let $h_i = 2^{-i}h_0$, $i=0,\ldots,l$, and 
$e_i = e_{h_i} = d_{h_i}(K_{\beta}^{h_i}, (H_{\beta}^{h_i})^{-1})$, 
$a_i = C\beta^{-1}(\gamma_u h_i + \gamma_p h_i)$. 
We have
\begin{equation}\label{equ:mle}
\begin{aligned}
  d_{h_i}\big(\mathcal{G}(K_{\beta}^{2h_i}), (H_{\beta}^{h_i})^{-1}\big)  
 &\leq d_{h_i}\big(\mathcal{G}(K_{\beta}^{2h_i}), 
   \mathcal{G}((H_{\beta}^{2h_i})^{-1})\big) \\
 & \quad +d_{h_i}\big(\mathcal{G}((H_{\beta}^{2h_i})^{-1}), 
                 (H_{\beta}^{h_i})^{-1}\big) \\
 &\overset{\eqref{equ:dist_prec2}}{\leq} 
  d_{h_i}\big(\mathcal{G}(K_{\beta}^{2h_i}), 
  \mathcal{G}((H_{\beta}^{2h_i})^{-1})\big) 
  +\frac{C}{\beta}(\gamma_u h_i^2 + \gamma_p h_i)\\ 
 &\leq d_{2h_i}\big(K_{\beta}^{2h_i},(H_{\beta}^{2h_i})^{-1}\big)  
 +\frac{C}{\beta}(\gamma_u h_i^2 + \gamma_p h_i) \\
& \leq e_{i-1} + a_i, \quad i = 1, \ldots, l, 
\end{aligned}
\end{equation}
where we have used the fact that 
\(d_{h_i}(\mathcal{G}_{h_i}(T_1),\mathcal{G}_{h_i}(T_2)) \leq 
d_{2h_i}(T_1,T_2)\), 
for any two symmetric positive definitive operators 
$T_1,T_2 \in \mathcal{L}(V_{2h_i})$ 
 \cite[Lemma 5.1]{DD}.

It is shown in \cite{DD} that for any $M,H \in \mathcal{L}_+(V_{h_i})$, 
with $d_{h_i}(M,H^{-1}) <0.4$, we have
\begin{align}\label{equ:mld}
 d_{h_i}(\mathcal{N}_{h_i}(M),H^{-1}) \leq 2 d_{h_i}(M,H^{-1})^2, 
\end{align}
for $h_i < h_0$. 
An inductive argument implies that $e_i \leq 0.2$ for all $i$, 
provided it holds for $e_{0}$ and 
$C\beta^{-1}(\gamma_u h_0^2 + \gamma_p h_0) < 0.1$. 
Thus, 
 combining \eqref{equ:mld} and \eqref{equ:mle} we obtain  
\begin{align*}
 e_i \leq 2(e_{i-1}+ a_i)^2, 
 \quad i = 1, \ldots l-1.
\end{align*}
Note that the hypothesis of Lemma 3.1 are satisfied with 
$f = 1/4$ which implies 
\begin{align*}
 e_i \leq 8 \frac{C^2}{\beta^2}(\gamma_u h_i^2 + \gamma_p h_i)^2, \quad 
 i = 1, \ldots, l-1. 
\end{align*}
In particular, we have
\begin{align*}
 e_{2h} \leq 32 \frac{C^2}{\beta^2}(\gamma_u h^2 + \gamma_p h)^2. 
\end{align*}
At the finest level, $i=l$, \eqref{equ:mle} becomes
\begin{align*}
 d_h\big(K_{\beta}^h, (H_{\beta}^h)^{-1}\big) \leq e_{2h} 
 + \frac{C}{\beta}(\gamma_u h^2 + \gamma_p h), 
\end{align*}
which combined with the above estimate yields the assertion of the theorem.
\end{proof}

Finally, let us note that 
an important difference between our method   
and the classical multigrid is that here
the base case needs to be chosen sufficiently fine, whereas in classical 
multigrid it can be as coarse as possible, in general.

\section{Numerical results}\label{sec:num}
In this section we present some numerical experiments  
that illustrate the application of the  preconditioner  
introduced in Section~\ref{sec:prec}. 

Let $\Omega = (0,1)^2$ and consider an optimal control problem of the 
form~\eqref{equ:cost}. We consider a family of uniform rectangular grids 
with mesh size $h$ and discretize the problem using Taylor-Hood  
$\mathbf{Q}_2-\mathbf{Q}_1$ elements for velocity-pressure and 
$\mathbf{Q}_2$ elements for the control. 
The problem was solved using \textsc{Matlab R2010a}.  
We perform three types of experiments: velocity control only ($\gamma_u=1, \gamma_p=0$), 
mixed velocity-pressure control ($\gamma_u=1, \gamma_p\ne 0$), and
pressure control only ($\gamma_u=0, \gamma_p=1$).

First, we summarize the numerical results obtained for ``in-vitro experiments".
As mentioned earlier, the Hessian matrix is dense, therefore it is never formed in practice, and the 
matrix-vector products in the preconditioned conjugate gradient (PCG) are 
implemented matrix-free.  
However, in order to evaluate directly the spectral distance between the 
Hessian and the proposed two-level preconditioner, we formed the matrices for 
moderate values of the mesh size $h$. In Table~\ref{table:spectrum} 
we present the joint spectrum analysis for $\beta=1$, $\gamma_u=1$, 
and $\gamma_p=0$ with 
$d_h=\max\{|\ln\alpha|: \alpha \in \sigma(H_{\beta}^h,T_{\beta}^h)\}$, 
where 
$\sigma(A,B)$ denotes the set of generalized eigenvalues of $A$,$B$.  
The results indicate optimal third-order convergence.
In Table \ref{table:spectrump} we present similar results for the case of 
pressure control only. In this case we observe an optimal 
quadratic convergence rate.  
We note that our computational results show a better behavior than predicted 
by Theorem~\ref{thm:prec}. We think this is due to the particular type of 
convex domain chosen here, for which the Stokes problem has better regularity 
than what was assumed in Theorem~\ref{thm:prec}, 
and also to the use of quadratic elements for the control. 

We also remark that the numerical estimates of $d_h$ in case of pressure control only
were obtained with a discretization that represents faithfully the finite element formulation
in Section~\ref{ssec:feapprox}. However, in practice, instead of using average-zero pressures
it is convenient to use a pressure space where  the pressures are set to zero at a fixed location,
 e.g., a corner. If we compute the operators $H_{\beta}^h$, $T_{\beta}^h$ using the latter 
space, we note that $\sigma(H_{\beta}^h, T_{\beta}^h)$ contains exactly two generalized
eigenvalues that are of size $O(1)$, instead of $O(h)$, as predicted by Theorem~\ref{thm:prec}.
If $\widetilde{\sigma}(H_{\beta}^h, T_{\beta}^h)$ is the subset of $\sigma(H_{\beta}^h, T_{\beta}^h)$
obtained after excluding the two generalized eigenvalues that are $O(1)$, and 
$\tilde{d}_h=\max\{|\ln\alpha|: \alpha \in \widetilde{\sigma}(H_{\beta}^h,T_{\beta}^h)\}$, 
then $\tilde{d}_h=O(h)$, as the theory predicts for $d_h$.

\junk{
\textit{Should mention different ways of implementing the average-free 
condition for pressure? - if you fix pressure at a point you get an outlier 
generalized eigenvalue whose value doesn't change with mesh size - if you 
ignore that eigenvalue, the distance behaves as expected. If you implement 
the zero-average condition then the spectral distance behaves as expected, 
however the time it takes to solves the Stokes system increases significantly, 
due to change in the sparsity pattern; also the Stokes matrix is not symmetric 
anymore}
} 
\begin{table}\centering
\caption{Joint spectrum analysis for $\beta=1$: velocity control only
         ($\gamma_u=1$, $\gamma_p=0$).}
\label{table:spectrum}
\begin{tabular}{|l|c|c|}
\hline
$h$       & $d_h$     &  $d_{2h}/d_h$ \\
\hline
$2^{-2}$ & $1.0274 \times 10^{-4}$ & N/A    \\
\hline
$2^{-3}$ & $1.3308\times 10^{-5}$ & 7.7205 \\
\hline
$2^{-4}$ & $1.3883 \times 10^{-6}$ & 9.5858 \\
\hline
$2^{-5}$ & $1.5834\times 10^{-7}$ & 8.7675 \\
\hline
\end{tabular}
\end{table}

\begin{table}[h!]\centering
\caption{Joint spectrum analysis for $\beta=1$: pressure control only
         ($\gamma_u=0$, $\gamma_p=1$).}
\label{table:spectrump}
\begin{tabular}{|l|c|c|}
\hline
$h$      & $d_h$ &  $d_{2h}/d_h$ \\
\hline
$2^{-2}$ &   $1.9613 \times 10^{-2}$  & N/A    \\
\hline
$2^{-3}$ &   $5.6686 \times 10^{-3}$    & 3.4599 \\
\hline
$2^{-4}$ &   $1.6703 \times 10^{-3}$    & 3.3937 \\
\hline
$2^{-5}$ &   $4.3735 \times 10^{-4}$    & 3.8192 \\
\hline
\end{tabular}
\end{table}

The next type of numerical experiments  
regard the solution of the control 
problem \eqref{equ:cost}. Specifically, we compare the number of iterations 
required to solve the linear system~\eqref{equ:sysr} 
with unpreconditioned CG to the case when  CG is used with a 
multilevel preconditioner with $1-4$ levels (depending on  resolution, see more comments below).
For the results presented here, we chose the target velocity 
$$\uu_d = (-2x^2y(1-x)^2(1-3y+2y^2), 2xy^2(1-y)^2(1-3x+2x^2))^T$$ 
and the  target pressure $$p_d = \cos \pi x \cos \pi y\ .$$ 
In Table~\ref{table:all} we summarize the results obtained for 
velocity control only (\mbox{$\gamma_u=1$}, $\gamma_p=0$), 
mixed velocity-pressure control (\mbox{$\gamma_u=1, \gamma_p = 10^{-5}, 10^{-4}, 10^{-3}$}), and
pressure control only ($\gamma_u=0, \gamma_p=1$); for each case 
a range of values for $\beta$ is chosen. 

The choice of the values for $\gamma_p$
in the mixed velocity-pressure control is justified by the data in Table~\ref{table:u}, where
we show for each case (and fixed resolution $h=1/32$) the relative error in the recovered data for the  velocity 
$E_{\vec{u}} = \| \vec{u}_h^{\min}-\vec{u}_d\|/\|\vec{u}_d\|$ 
and pressure $E_p=\|p_h^{\min}-p_d\|/\|p_d\|$. As can be seen from Table~\ref{table:u}
for pure velocity control, the pressure is not recovered at all ($E_p \approx 1$), while
for pure pressure control the velocity is not recovered ($E_{\vec{u}} \approx 1$). As expected, for mixed control
both pressure and velocity are being recovered in the sense that
both $E_{\vec{u}}$ and $E_p$ decrease with $\beta\downarrow 0$. However, for $\gamma_p=10^{-5}$
the relative velocity error $E_{\vec{u}}$ is one order of magnitude smaller than $E_p$, so velocity 
is better recovered than pressure. For $\gamma_p=10^{-4}$, $E_{\vec{u}}$ and  $E_p$ are of comparable size
(within a factor of 2), and if $\gamma_p=10^{-3}$, then the situation is reversed with  $E_p$ being  one order
of magnitude smaller than~$E_{\vec{u}}$.

We now return to the actual results in Table~\ref{table:all}. First note that
for all cases unpreconditioned CG solves the system~\eqref{equ:sysr} in a number
of iterations that appears to be almost mesh-independent, and is clearly bounded 
with $h\downarrow 0$. While this is certainly consistent with our remark from Section~\ref{sec:conv}
regarding the mesh independence of the condition number of $H_{\beta}^h$, the relatively
low number of CG iterations is due to the fact that the continuous counterpart 
of $H_{\beta}^h$ is compact, which implies that the number of eigenvalues of $H_{\beta}^h$ which
are away from $\beta$ is small. Also in accordance with~\eqref{equ:condno}, 
the number of unpreconditioned CG iterations increases with
$\beta\downarrow 0$, which is expected since the number of relevant eigenvalues
of $H_{\beta}^h$ increases as $\beta\downarrow 0$.
For velocity control only (the top part of Table~\ref{table:all}) 
we observe a significant reduction in the number of iterations when multilevel 
preconditioners are used, as well as a decrease in the number of iterations 
with mesh size.   
For example, for $\beta=10^{-7}$ and $h=2^{-8}$, the number of preconditioned 
iterations with a four-level preconditioner is significantly smaller than 
in the unpreconditioned case, i.e, $3$ iterations vs. $78$ iterations. 
Although a preconditioned iteration is more expensive than an 
unpreconditioned one, for large problems the overall cost of the preconditioned 
solver is much lower than  of the unpreconditioned one, as can be seen from Table~\ref{table:timeu}.
At the other end of these experiments (bottom of Table~\ref{table:all}) we see the 
cases of pressure control only. We also see the decrease in number of iterations with
$h\downarrow 0$ when comparing the behavior of, say, 3-grid preconditioners at different resolutions:
for example, for $\gamma_u=0, \gamma_p=1, \beta=10^{-3}$ the number of 3-level preconditioners dropped from 14 ($h=2^{-6}$)
to 10 ($h=2^{-7}$); similar results are consistently observed throughout Table~\ref{table:all}. However,
for the case of pressure control only, the drop in number of iterations from unpreconditioned CG to 
multilevel preconditioned CG is not as dramatic as for velocity control only. As can be inferred
from Table~\ref{table:all}, the efficiency of the multilevel preconditioned CG versus unpreconditioned CG
measured as the ratio  of the number of iterations gradually decreases with the increase of the ratio
$\gamma_p/\gamma_u$ (for otherwise comparable experiments), as predicted by Theorems~\ref{thm:prec} and~\ref{thm:multi}.

In our implementation we have used direct methods for solving 
the Stokes system, and we actually constructed the base-case 
Hessian and stored its inverse. These choices have limited
our computations to $h=2^{-8}$ and $h_{\mathrm{base}}=2^{-5}$, which is why
for $h=2^{-7}$ we were unable to test the two-grid preconditioner (this would 
have required saving a base-case Hessian for $h=2^{-6}$), and for $h=2^{-8}$
we were only able to compute using a four-level preconditioner. While we 
already commented on the positive side of the numerical results, it is worth
noting the pitfalls: if the coarsest grid is too coarse, then the quality of the 
preconditioner declines to the point that it is hurting the computation. This
can be seen in the groups of columns and rows in Table~\ref{table:all}
corresponding to $h=2^{-6}$ and
$\beta=10^{-6}, 10^{-7}$: the use of  too many levels results in a spike in
the number of iterations to the point of non-convergence (within a maximum
number of 100 iterations allowed).

Finally, we would like to comment on the robustness of our algorithm
with respect to the accuracy of the Stokes solve. For large-scale 
problems the Stokes system on the  finer grids is expected to be solved using iterative rather than
direct methods, which reduces the accuracy of computing matrix-vector multiplications for the Hessian matrix $H_{\beta}^h$.
We have repeated our numerical experiments in that, except for the coarsest scale, 
we replaced the direct solve of the Stokes systems with a preconditioned MINRES solve, as described 
in~\cite{MR2155549}. For the case of pure velocity control ($\gamma_p=0$) we found no significant change in the number
of iterations in Table~\ref{table:all}. However, for the case of mixed- and pure pressure control ($\gamma_p\ne 0$) the quality
of our algorithm appeared to decline significantly. We identified as the primary cause for this behavior
the fact that, even when the velocity variables are well resolved, that is, the relative error between
the solutions obtained via direct vs. iterative methods is on the order of $10^{-8}$, the relative error
in the pressure terms can be quite high ($10^{-2}$--$10^{-4}$). Since our algorithm relies
on the ability to compute the operator $\op{P}_h$ with sufficient accuracy,
we are not able at this point to draw conclusions with respect to the influence 
of using iterative methods on our algorithm for mixed- and pure pressure control.

\begin{table}[h!]\centering
\caption{Iteration count for multilevel preconditioners; ``nc'' means ``not-converged''. The tolerance is
set at $10^{-12}$.}
\label{table:all}
\begin{tabular}{|c||c|c|c|c||c|c|c||c|c||}
 \hline
 $h$ & 
    \multicolumn{4}{|c||}{$2^{-6}$} &
    \multicolumn{3}{|c||}{$2^{-7}$} &
    \multicolumn{2}{|c||}{$2^{-8}$}  \\ \hline
 num. levels & 1 & 2 & 3 & 4
            & 1 & 3 & 4
            & 1 & 4 \\ \hline \hline
 \multicolumn{10}{|c||}{$\gamma_u=1,\ \ \gamma_p=0$}\\ \hline
$\beta=10^{-4}$ & 7  & 3 & 3 & 3
                 & 7  & 2 & 2
                 & 7   & 2 \\ 
 $\beta=10^{-5}$  & 13  & 3 & 3 & 3
                 & 13  & 3 & 3
                 & 14  & 2 \\ 
 $\beta=10^{-6}$  & 29  & 4 & 4 & 6
                 & 29   & 3 & 3
                 & 32   & 3 \\ 
 $\beta=10^{-7}$  & 74  & 5 & 7 & 62
                 & 75   & 3 & 5
                 & 78   & 3 \\ 
 \hline\hline
\multicolumn{10}{|c||}{$\gamma_u=1,\ \ \gamma_p=10^{-5}$}\\ \hline
$\beta=10^{-4}$  & 10  & 4 & 5 & 5
                 & 10  & 5 & 5
                 & 11   & 5 \\ 
 $\beta=10^{-5}$ & 20  & 5 & 6 & 8
                 & 20  & 6 & 8
                 & 22  & 7 \\ 
 $\beta=10^{-6}$ & 45  & 7 & 8 & 13
                 & 44   & 7 & 9
                 & 45   & 9 \\ 
 $\beta=10^{-7}$  & 112  & 9 & 14 & nc
                 & 113   & 9 & 14
                 & 119   & 11 \\ 
 \hline \hline
\multicolumn{10}{|c||}{$\gamma_u=1,\ \ \gamma_p=10^{-4}$}\\ \hline
$\beta=10^{-4}$  & 10  & 5 & 6 & 8
                 & 11  & 5 & 7
                 & 11   & 7 \\ 
 $\beta=10^{-5}$ & 21  & 7 & 7 & 9
                 & 23  & 7 & 9
                 & 24  & 9 \\ 
 $\beta=10^{-6}$  & 48  & 8 & 10 & 16
                 & 48   & 9 & 13
                 & 49   & 11 \\ 
 $\beta=10^{-7}$ & 122  & 13 & 17 & nc
                 & 126   & 11 & 22
                 & 133   & 20 \\ 
 \hline \hline
\multicolumn{10}{|c||}{$\gamma_u=1,\ \ \gamma_p=10^{-3}$}\\ \hline
$\beta=10^{-4}$  & 12  & 6 & 7 & 9
                 & 12  & 7 & 9
                 & 13   & 9 \\ 
 $\beta=10^{-5}$ & 25  & 8 & 9 & 13
                 & 26  & 9 & 13
                 & 28  & 11 \\ 
 $\beta=10^{-6}$ & 61  & 11 & 17 & 33
                 & 63   & 12 & 22
                 & 65   & 20 \\ 
\hline  \hline
\multicolumn{10}{|c||}{$\gamma_u=0,\ \ \gamma_p=1$}\\ \hline
$\beta=10^{-1}$  & 8 & 5 & 7 & 9
                 & 8   & 6 & 8
                 & 8  & 8 \\ 
 $\beta=10^{-2}$ & 13  & 7 & 9 & 12
                 & 13   & 8 & 11
                 & 13   & 11 \\ 
 $\beta=10^{-3}$ & 27  & 10 & 14 & nc
                 & 26   & 10 & 15
                 & 27  & 14 \\ 
\hline
\end{tabular}
\end{table}

\begin{table}[h!]\centering
\caption{Relative error for recovered data for velocity control 
$E_{\vec{u}} = \| \vec{u}_h^{\min}-\vec{u}_d\|/\|\vec{u}_d\|$ 
and pressure $E_p=\|p_h^{\min}-p_d\|/\|p_d\|$.}
\label{table:u}
\begin{tabular}{|c|c||c|c|c|c|c|c||}
 \hline
$\gamma_u$ & $\gamma_p$ & 
$E_{\vec{u}}$ & $E_p$ & 
$E_{\vec{u}}$ & $E_p$ & 
$E_{\vec{u}}$ & $E_p$\\ \hline \hline
\multicolumn{2}{|c||}{$\beta$} & \multicolumn{2}{|c}{$10^{-5}$} &
\multicolumn{2}{|c}{$10^{-6}$} &\multicolumn{2}{|c||}{$10^{-7}$} \\\hline
$1$ & $0$ &
 4.07e-2& $\approx 1$ &
 8.23e-3& $\approx 1$ &
 1.88e-3& $\approx 1$ \\
$1$ & $10^{-5}$ &
 4.32e-2 & 4.32e-1&
 1.65e-2 & 2.80e-1&
 8.28e-3 & 7.05e-2\\
$1$ & $10^{-4}$ &
 6.66e-2 & 2.75e-1&
 3.30e-2 & 6.85e-2& 
 9.55e-3 & 8.44e-3\\
$1$ & $10^{-3}$ &
 1.23e-1 & 6.56e-2&
 3.78e-2 & 8.18e-3& 
 9.70e-3 & 8.98e-4\\
\hline \hline
\multicolumn{2}{|c||}{$\beta$} & \multicolumn{2}{|c}{$10^{-1}$} &
\multicolumn{2}{|c}{$10^{-2}$} &\multicolumn{2}{|c||}{$10^{-3}$} \\\hline
$0$ & $1$ &
 1.02& 2.66e-1 &
 1.07& 6.12e-2 &
 1.09& 7.34e-3 \\
 \hline
\end{tabular}
\end{table}

\begin{table}[h!]\centering
\caption{Velocity control only: time comparison for $h=2^{-8}$. 
 No. state variables: 588290, no. control variables: 522242.}
\label{table:timeu}
\begin{tabular}{|c|c|c|}
\hline 
 no. levels      &   1    & 4     \\ \hline
 $\beta=10^{-6}$ & 3613 s & 493 s + 1486 s for base case\\
 $\beta=10^{-7}$ & 8953 s & 575 s + 1492 s for base case  \\ 
\hline
\end{tabular}
\end{table}


\section{Conclusions}
\label{sec:concl}
In this article we  introduced Schur-based two- and multigrid preconditioners for the KKT system associated to
a Tikhonov-regularized  optimal control problem constrained by the Stokes system. We showed that, if the Stokes-system is discretized 
using  a stable pair of finite elements, the preconditioner approximates the  
reduced Hessian of KKT system to optimal order with respect to the convergence
order of the finite element method. As a  consequence, the number of
preconditioned CG iterations needed for solving the optimal control problem to a given tolerance
decreases with increasing resolution, asymptoting to just one iteration as $h\downarrow 0$. The problem
discussed in this article forms an important stepping stone towards finding highly efficient
methods for solving large-scale optimal control problems constrained by the Navier-Stokes system, 
which are the focus of our current research.

\appendix
\renewcommand{\thesection}{A}
\section{Some facts about spectral distance} \label{sec:app}

Let $(X, \langle \cdot,\cdot\rangle)$ be a real finite dimensional Hilbert 
space and denote the complexification of $X$  by 
$X^{\mathbb{C}} = \{u + \mathbf{i} v: u,v \in X\}$.
Let $\mathcal{L}_{+}(X) = \{T \in \mathcal{L}(X): \langle Tu,u\rangle >0,  
\ \forall u \in X \setminus \{0\} \}$ 
and define the spectral distance between $S,T \in \mathcal{L}_+(X)$ to be
\begin{align*}
 d_X(S,T) =\underset{w \in X^{\mathbb{C}} \setminus \{0\}}{\sup}
                \Bigg|\ln \frac{(S^{\mathbb{C}} w,w)}
                               {(T^{\mathbb{C}} w,w)}\Bigg|, 
\end{align*}
where $T^{\mathbb{C}}(u+\mathbf{i}v) = T(u) + \mathbf{i} T(v)$ is the 
complexification of $T$.
The following inequalities were proved in \cite[Lemma 3.2]{DD}:
\begin{lemma}\label{lma:log_ineq}
If $\alpha\in(0,1)$ and $z\in {\mathcal B}_{\alpha}(1)$, then
\begin{equation}\label{eq:log_ineq}
 \frac{\ln(1+\alpha)}{\alpha}|1-z| \le |\ln z| \le
 \frac{|\ln(1-\alpha)|}{\alpha} |1-z|\ .
\end{equation}
For $\abs{\ln z}\le \delta$ we have
\begin{equation}\label{eq:log_ineq2}
 \frac{1-e^{-\delta}}{\delta} |\ln z|\le |1-z| \le 
 \frac{e^{\delta}-1}{\delta} |\ln z| .
\end{equation}
\end{lemma}

\begin{lemma}\label{lemma:app}
Let $L, H\in \op{L}_+(\op{X})$ such that
$$\min\left(d_X(L^{-1}, H), d_X(L, H^{-1})\right)\le \delta\ .$$
Then
\begin{equation} \label{eq:spec_rad_vs_spec_dist}
 \rho(I-L H)\le \frac{e^{\delta}-1}{\delta} 
 \min\left(d_X(L^{-1}, H), d_X(L, H^{-1})\right)\ .
\end{equation}
\end{lemma}
\begin{proof}
If $\lambda\in\sigma(I-L H)$ then there exists a unit vector 
$u\in X^{\mathbb{C}}$ such that $(I-L H)u = \lambda u$, therefore
\begin{eqnarray} \label{eq:srd0}
 (1-\lambda)u = L H u\ .
\end{eqnarray}
After left-multiplying  with $L^{-1}$ and taking the inner product with $u$ 
we obtain
$$(1-\lambda)\innprd{L^{-1}u}{u} = \innprd{H u}{u},\ \ \mathrm{therefore}\ \ 
\lambda = 1-\frac{\innprd{H u}{u}}{\innprd{L^{-1} u}{u}}\ .$$
If we substitute $v=H^{-1}u$ in~\eqref{eq:srd0} and take the inner product 
with $v$ we have
$$
 (1-\lambda)H^{-1}v = L v\ ,\ \ \mathrm{therefore}\ 
 \lambda = 1-\frac{\innprd{L v}{v}}{\innprd{H^{-1} v}{v}}\ .
$$
Hence, if $d_{X}(L^{-1}, H)\le \delta$, then
\begin{eqnarray*}
 \rho(I-L H) &\le& \sup\{|1-z|\ :\ z = \innprd{H u}{u}/\innprd{L^{-1} u}{u}\ 
                   \mathrm{for\ some\ }u\in X^{\mathbb{C}}\setminus\{0\}\}\\
             &\stackrel{\eqref{eq:log_ineq2}}{\le}
             &\frac{e^{\delta}-1}{\delta} d_{X}(L^{-1}, H)\ .
\end{eqnarray*}
Instead, if $d_{X}(L, H^{-1})\le \delta$, then
\begin{eqnarray*}
 \rho(I-L H)&\le& \sup\{|1-z|\ :\ z = \innprd{L u}{u}/\innprd{H^{-1} u}{u}\ 
                  \mathrm{for\ some\ }u\in X^{\mathbb{C}}\setminus\{0\}\}\\
            &\stackrel{\eqref{eq:log_ineq2}}{\le}
            &\frac{e^{\delta}-1}{\delta} d_{X}(L, H^{-1})\ .
\end{eqnarray*}
which proves~\eqref{eq:spec_rad_vs_spec_dist}.
\end{proof}

\bibliography{reference_mg}
\bibliographystyle{siam}

\end{document}